\documentclass[10pt, leqno]{amsart}
 \usepackage{graphics}
\usepackage{amsfonts}
\usepackage{graphicx,nicefrac}
\usepackage{mathrsfs,dsfont}
\usepackage{amsmath,amstext,amsthm,amssymb}

\newtheorem{Theorem}{Theorem}[section]
\newtheorem{Proposition}[Theorem]{Proposition}

\theoremstyle{definition}

\title{On Lelong Numbers of Positive Closed Currents on $\mathbb{P}^n$}
\author{James J. Heffers}
\begin{document}

\maketitle

\author

\begin{center}

\textit{Syracuse University \\
215 Carnegie Building, Syracuse University, Syracuse N.Y. 13422\\
e-mail: jjheffer@syr.edu}

\end{center}

\begin{abstract}
\noindent Let $T$ be a positive closed current of bidimension $(p,p)$ with unit mass on the complex projective space $\mathbb P^n$. For certain values of $\alpha$ and $\beta = \beta(p, \alpha)$ we show that if $T$ has enough points where the Lelong number is at least $\alpha$, then the upper level set $E_{\beta}^+ (T)$ of points where $T$ has Lelong number strictly larger than $\beta$ has certain geometric properties.  
\end{abstract}

\section{Introduction}

Let $T$ be a positive closed current of bidimension $(p,p)$ on $\mathbb{P}^n$ which has mass $\|T\|=1$, where
$$\|T\| := \int_{\mathbb{P}^n} T \wedge \omega_n^p $$

 \noindent and $\omega_n$ is the Fubini-Study form on $\mathbb{P}^n$.  We consider the following upper level sets of Lelong numbers $\nu( T, q)$ of the current $T$

\begin{center}

$E_{\alpha} (T) = \{q \in \mathbb{P}^n \, | \, \nu(T,q) \geq \alpha \},$

$E_{\alpha}^+ (T) = \{q \in \mathbb{P}^n \, | \, \nu(T,q) > \alpha \}.$

\end{center}

It has been shown by Siu \cite{Siu74} that $E_{\alpha} (T)$ is an analytic subvariety of dimension at most $p$ when $\alpha >0$.  Our goal is to gain more understanding of the geometric properties of these upper level sets by attempting to generalize some of the results for bidimension $(1,1)$ currents laid out in \cite{C06} and \cite{H17} to analogous results for bidimension $(p,p)$ currents by utilizing the tools given to us by Coman and Truong in \cite{CT15}.  We start first by generalizing \cite[Theorem 3.10]{C06}, which states that given a bidimension $(1,1)$ positive closed current $T$,  $\alpha >\frac{1}{2}$ , $\beta = (2-\alpha)/3$, and two points $q_1, q_2 \in \mathbb{P}^n$ such that $\nu(T,q_i)\geq\alpha$, then $E_{\beta}^+(T)$ can be contained in a line, with the exception of at most one point.  In doing so we find that $\beta$ depends on both $p$ and $\alpha$ to get the following:

\begin{Theorem}\label{T:mt1} Let $T$ be a positive closed current of bidimension $(p,p)$ on $\mathbb{P}^n$, $0< p <n$, $\|T\|=1$, $\alpha > \frac{p}{p+1}$ and $\beta = \frac{p^2 + p - \alpha}{p(p+2)}$.  Let $q_1, q_2$ be points in $\mathbb{P} ^{n}$ such that $\nu (T, q_{i}) \geq \alpha$.  Then either $E_{\beta}^+(T)$ is contained in a $p$-dimensional linear subspace or there exists a complex line $L$ such that $| E_{\beta }^+ (T) \backslash L | = p$. 
\end{Theorem}

\smallskip

The lower bound on $\alpha$ and the fact that we allow for $p$ points to be omitted from the line, while seemingly an arbitrary choice, comes from the conclusion of \cite[Theorem 1.2]{CT15}, stated in the next section.  At the end of the third section, we will investigate two examples to show that this $\beta$ value is sharp for this property, and that the assumption of needing two points $q_1, q_2$ where the current has ``large" Lelong number is necessary.  We also will generalize \cite[Theorem 3.12]{C06}, in which Coman showed that for a bidimension $(1,1)$ current $T$ and $\alpha \geq 1/2$, if the set $E_{1-\alpha}$ contained three collinear points on some line $L$, then $E_\alpha^+(T)$ is contained on $L$ with the exception of at most one point.

\smallskip

We then turn our attention to generalizing \cite[Proposition 3.6]{H17}, which was originally proved only for bidimension $(1,1)$ currents on $\mathbb{P}^2$, by proving the following:

\begin{Theorem}\label{P:33}  Let $T$ be a positive closed current of bidimension $(1,1)$ on $\mathbb{P}^n$ with $\|T\|=1$. Let $\{ q_{i} \}_{i=1}^{4}$ be collinear points in $\mathbb{P} ^{n}$ and $\nu (T, q_{i}) \geq \alpha > 2/5$.  Let $\beta = \frac{2}{3} (1- \alpha)$.  Then there exists two lines $L_1$, $L_2$ such that $|E_{\beta}^+(T)\backslash (L_1\cup L_2)| \leq 1$.
\end{Theorem}

We close by looking at a weak generalization of \cite[Theorem 1.1]{H17} from $\mathbb{P}^2$ to $\mathbb{P}^n$, and making some remarks on the difficulties of attempting to make a stronger result.

\section{Preliminaries}

In this section we will lay out the tools that will be commonly used in the upcoming proofs for the convenience of the reader.  To show Theorem 1.1, we will rely upon the following:

\begin{Theorem}\cite[Theorem 1.2]{CT15}\label{T:mt1} If $T$ is a positive closed current of bidimension $(p,p)$ on $\mathbb{P}^n$, $0<p<n$, with $\|T\| = 1$, then the set $E_{p/(p+1)}^+(T,\mathbb{P}^n)$ is either contained in a $p$-dimensional linear subspace of $\mathbb{P}^n$ or else it is a finite set and $|E_{p/(p+1)}^+(T,\mathbb{P}^n)\backslash L| = p$ for some line $L$. 
\end{Theorem}

We will also make specific use of the previous theorem when $p=1$, which was shown by Coman:

\begin{Theorem}\cite[Theorem 1.1]{C06} \label{T:21} Let $T$ be a positive closed current of bidimension $(1,1)$ with $\|T\|=1$ on $\mathbb{P}^n$.  If $\alpha \geq \frac{1}{2}$ then there exists a line $L$ such that either $E_{\alpha}^+(T) \subset L$, or $E_{\alpha}^+(T)$ is a finite set and $|E_{\alpha}^+ (T) \backslash L| \leq 1$.  Moreover, if $\alpha \geq 2/3$ then $E_{\alpha}^+ (T) \subset L$.
 \end{Theorem}

We will also be working on generalizing the following result to $\mathbb{P}^n$:

\begin{Theorem}\cite[Theorem 1.1]{H17}\label{T:mt1} Let $T$ be a positive closed current of bidimension $(1,1)$ on $\mathbb{P}^2$ with $\|T\|=1$, $\alpha > 2/5$ and $\beta = \frac{2}{3} (1- \alpha)$.  Let $\{ q_{i} \}_{i=1}^{4}$ be points in $\mathbb{P} ^{2}$ such that $\nu (T, q_{i}) \geq \alpha$.  Then there exists a conic $C$ (possibly reducible) such that $| E_{\beta }^+ (T) \backslash C | \leq 1$.
\end{Theorem}

The proof of the above theorem utilized \cite[Theorem 1.2]{C06}, but for bidimension $(1,1)$ currents on $\mathbb{P}^n$, the only analogous theorem to assist us is the following:

\begin{Theorem}\cite[Theorem 3.3]{CT15}\label{T:mt1} Let $T$ be a positive closed current of bidimension $(1,1)$ on $\mathbb{P}^n$ with $\|T\|=1$.  If $|E_{2/5}^+(T)| > 37$ then there exists a curve $C\subset \mathbb{P}^n$ of degree at most $2$ such that $|E_{2/5}^+(T)\backslash C | \leq 1$.
\end{Theorem}

We will also make use of the next proposition which follows easily from Demailly's regularization theorem \cite[Proposition 3.7]{D92}.

\begin{Proposition} \label{P:25}  Let $R$ be a positive closed current of bidegree $(1,1)$ on $\mathbb{P}^n$, $\nu (R, x_i) > a_i$, $i = 1, \dots , N$ for $x_i \in \mathbb{P}^n$ and $a_i > 0$.  Then there exists a positive closed bidegree $(1,1)$ current $R'$ on $\mathbb{P}^n$ with analytic singularities such that $\|R'\| = \|R\|$, $\nu(R',x_i) > a_i$ for $i = 1, \dots , N$, and $\nu(R' , x) \leq \nu(R,x)$ for all $x\in \mathbb{P}^n$.  In particular, $R'$ is smooth in a neighborhood of every point where $R$ has $0$ Lelong number.
\end{Proposition}

\section{Proof of Theorem 1.1}  

To start with, let us recall some basic definitions.  Consider $A = \{x_1, \dots, x_{p+1}\}$, $x_i \in \mathbb{P}^n$.  By the Span $(A)$, we mean the smallest linear subspace $V\subset \mathbb{P}^n$ that contains $A$.  If $p\leq n$ and span$(A)$ is a $p$-dimensional space, then we say $\{x_i\}_{i=1}^{p+1}$ are linearly independent.  If we have $p>n+1$ points, then we say they are in general position if any $n+1$ of them are linearly independent.  Assume $\|T\| = 1$ and we now prove Theorem 1.1.

\begin{proof}

Suppose $\{x_i\}_{i=1}^{p}$ are points in $E_{\beta}^+(T)\backslash \{q_1, q_2\}$, and let $A := \{q_1, q_2, x_1, \dots ,x_{p-1}\}$ and then $V_1 = \text{span}(A)$. Suppose that $\{q_1, q_2\}\cup \{x_i\}_{i=1}^{p-1}$ are linearly independent and then $V_1 := \text{span}(A)$ is a $p$ dimensional linear subspace and $x_p \in E_{\beta}^+(T)\backslash V_1$, noting that if no such points exists, then there is nothing to prove.  Let $L_1$ be the line spanned by $q_1,q_2$, and since the points in $A$ are linearly independent, $L_1$ does not contain any other points of $A$ or $x_p$. Note if $E_{\beta}^+(T) \backslash( A \cup \{x_p\}) = \emptyset$ then $|E_{\beta}^+(T)\backslash L_1 | = p$, and we are done.  Suppose then that $x_{p+1}\in E_{\beta}^+(T) \backslash (A\cup L_1)$, $x_p \neq x_{p+1}$.  Let $V_2$ be a $p$-dimensional linear space containing $\{x_i\}_{i=1}^{p+1}$ (observe $V_2$ need not be the only such $p$-dimensional  linear subspace containing these points).  Choose $\alpha '$ such that $\frac{p}{p+1} < \alpha ' < \alpha $ and $\nu(T, x_i) > \frac{p^2 + p - \alpha '}{p(p+2)}$, and define the current $R$ as follows:  

$$R := \frac{(p+1)\alpha ' - p}{(p+1)\alpha '} [V_2] + \frac{p}{(p+1)\alpha '} T.$$

\noindent Note that $\|R\| = 1$ as well as

$$\nu(R, q_i) \geq \frac{p}{(p+1)\alpha'}\nu(T,q_i)> \frac{p}{p+1} \quad i=1,2$$

\noindent and

$$\nu(R, x_i) > \frac{(p+1)\alpha ' - p}{(p+1)\alpha '} + \frac{p}{(p+1)\alpha '} \frac{p^2 + p - \alpha '}{p(p+2)}$$

$$=\frac{(p+2)(p+1) - 1}{(p+2)(p+1)} - \frac{p}{(p+2)(p+1)\alpha'} > \frac{p}{p+1}, \quad i=1,\dots ,p+1.$$

\smallskip

Thus by Coman \cite[Theorem 1.2]{CT15}, it must be the case that there is a line containing three points of $\{q_1, q_2, x_1, \dots, x_{p+1}\}$, say $L_2$.  We now have to break our argument into two cases, depending on if $x_{p+1}$ is contained in $V_1$ or not.

\textit{Case 1:}  Suppose that $x_{p+1}\notin V_1$.  Note that $L_2$ cannot contain $3$ points of the set $A$ as those points are linearly independent.  Thus it must be the case that $L_2$ contains both $x_p, x_{p+1}$ as otherwise if $L_2$ only contains one of them, the other two points would be from the set $A$, which means either $x_p$ or $x_{p+1}$ would be in the span of $A$, which is a contradiction.  So $L_2$ contains $x_p, x_{p+1}$ and some $y \in A$.  Now note we can find a new point $x_{p+2} \in  E_{\beta}^+(T) \backslash (A\cup L_2)$, as otherwise $| E_{\beta}^+(T) \backslash L_2| = p$ and we would be done.  Let $B:=\{x_i\}_{i=1}^{p+2}$, and let $U_i$ be a $p$-dimensional linear space containing $B\backslash \{x_i\}$. Choose $\alpha '$ such that $\frac{p}{p+1} < \alpha ' < \alpha $ and $\nu(T, x_i) > \frac{p^2 + p - \alpha '}{p(p+2)}$ for $i=1, \dots, p+2$. We now consider a new current $S$ given by:

$$S := \frac{(p+1)\alpha ' - p}{(p+1)(p+2)\alpha '}\sum_{i=1}^{p+2} [U_i] + \frac{p}{(p+1)\alpha '} T$$

\noindent note that $\|S\| = 1$ as well as

$$\nu(S, q_i) > \frac{p}{p+1} \quad i=1,2$$

\noindent and

$$\nu(S, x_i) > \frac{(p+1)\alpha ' - p}{(p+1)(p+2)\alpha '}(p+1) + \frac{p}{(p+1)\alpha '} \frac{p^2 + p - \alpha '}{p(p+2)}$$
$$=\frac{(p+1)^2\alpha' - (p+1)p +p^2 + p - \alpha'}{(p+1)(p+2)\alpha'} = \frac{p}{p+1},\quad i=1,\dots ,p+2.$$

\noindent Thus by \cite[Theorem 1.2]{CT15}, there exists a complex line $L_3$ that will contain four points of $A\cup \{x_p, x_{p+1}, x_{p+2}\}$.  As the points in $A$ are in general position, $L_3$ can only contain at most two points of $A$.  If $L_3$ contains two points of $A$, then $L_3$ will also contain at least one of $x_p$ or $x_{p+1}$ which means that at least one of $x_p, x_{p+1}$ lies in $V_1$, which is impossible.  So $L_3$ can only contain one point of $A$ which means $L_3$ contains $x_p, x_{p+1}, x_{p+2}$.  But now that means $L_3 = L_2$, and this is a contradiction as $x_{p+2} \notin L_2$.  So no such point $x_{p+2}$ can exist, and $L_2$ is the line that satisfies the conclusion of the theorem.

\textit{Case 2:}  Suppose that $x_{p+1}\in V_1$.  As $L_2$ must contain three points of $A\cup \{x_p, x_{p+1}\}$, it must be the case that $L_2$ contains $x_{p+1}, x_j, x_k$ (note that $x_j, x_k$ may actually be the $q_i$, but that is irrelevant) as the points in $A$ are in general position and $x_p\notin \text{span}(A)$.  After reindexing, say that $L_2$ contains $x_1, x_2, x_{p+1}$.  Arguing as we did in the first case, we can find a new point $x_{p+2} \in E_{\beta}^+(T)\backslash L_2$ and a line $L_3$ containing four points of $A\cup \{x_p, x_{p+1}, x_{p+2}\}$. Note $L_3 \subset V_1$.  As the points of $A$ are in general position, and since $x_{p+2}\notin L_2$, it must be the case that $L_3$ contains $x_{p+1},x_{p+2}$ and (after reindexing) $x_3, x_4$.  But now observe that $L_2$ is the line that spans $x_1, x_{p+1}$, $L_3$ is the line that spans $x_3, x_{p+1}$, and $L_2\cap L_3 \neq \emptyset$, so we have a $2$-dimensional linear space containing $x_1, x_2, x_3$ and $x_4$, which is a contradiction as they are linearly independent.  Thus we cannot have such a point $x_{p+2}$, and $L_2$ is the desired line that satisfies the conclusion of our theorem.

\end{proof}

\noindent\textbf{Remark:} When $p=1$, we have that $\alpha > 1/2$ and that $\beta = (2-\alpha)/3$, which is exactly the version proved by Coman, \cite[Theorem 3.10]{C06}.

\bigskip

Consider now a positive closed bidimension $(p,p)$ current $T$ on $\mathbb{P}^n$.  We now show that if $T$ has a ``small" Lelong number at a sufficient number of points in a $p$-dimensional linear subspace $V$, then either $E_{\alpha}^+(T) \subset V$ or that the line $L$ satisfying the conclusion of Theorem 2.1 is contained in $V$.

\begin{Theorem}\label{T:mt1} Let $T$ be a positive closed current of bidimension $(p,p)$ on $\mathbb{P}^n$ with $\|T\|=1$, $\alpha \geq \frac{p}{p+1}$ and suppose there are points $x_1, \dots, x_{p+2} \in E_{1 - \frac{\alpha}{p}}$ such that $\{x_i\}_{i=1}^{p+2}$ span a $p$-dimensional linear subspace $V$.  Then either $E_{\alpha}^+(T)$ is contained in $V$ or there exists a complex line $L\subset V$ such that $| E_{\alpha }^+ (T) \backslash L | = p$. 
\end{Theorem}

\begin{proof}

Let $\{x_i\}$ and $V$ be as stated above, and suppose there exists a point $q_1\in E_{\alpha}^+(T)\backslash V$, noting if no such point exists, we are already done.  Choose $\alpha'>\alpha$ such that $\nu(T,q_1) > \alpha'$, and thus $\nu(T,x_i) \geq 1-\frac{\alpha}{p} > 1-\frac{\alpha'}{p}$.  We consider now the current $R$ given by

$$R = \frac{(p+1)\alpha' -p}{(p+1)\alpha'} [V] + \frac{p}{(p+1)\alpha'} T$$

\noindent and observe that $\|R\| = 1$, $\nu(R,q_1) > \frac{p}{p+1}$ and

$$\nu(R, x_i) >  \frac{(p+1)\alpha' -p}{(p+1)\alpha'} + \frac{p - \alpha'}{(p+1)\alpha'} =\frac{p}{p+1}.$$

\noindent Thus by \cite[Theorem1.2]{CT15}, either $\{q_1, x_1, \dots, x_{p+2}\}$ are in a $p$-dimensional linear subspace or there is a line $L$ such that $|E_{\frac{p}{p+1}}^+ (R)\backslash L| = p$.  Since $\{x_i\}_{i=1}^{p+2}$ uniquely define $V$, and $q_1\notin V$, it must be the case that there is a line $L$ containing exactly $3$ of the $p+3$ points.  As $L$ must contain two points in $V$, it must be the case that $L\subset V$ and thus $q_1\notin L$, and say after reindexing that $L$ contains the points $x_1, x_2,$ and $x_3$.

We now show via contradiction that $L$ is the line satisfying the conclusion of the theorem.  Suppose there is $q_2\in E_{\alpha}^+(T)\backslash L$, $q_2 \neq q_1, x_i$.  Choose $\alpha' > \alpha$ such that $\nu(T, q_1) > \alpha'$ and $\nu(T,q_2) > \alpha'$.  Using the current $R$ as above we get $\nu(R,q_i)>\frac{p}{p+1}$ and $\nu(R,x_i)>\frac{p}{p+1}$.  Again we apply \cite[Theorem 1.2]{CT15} and since we cannot contain the $p+4$ points in $V$, we get that there is a line $L_1$ that must contain $4$ of the $p+4$ points. As two of those points must be in $V$, $L_1\subset V$ and thus $L_1$ must contain at least three of the $x_i$ points.  As the $x_i$ span $V$, the only three collinear points in $V$ are $x_1,x_2,x_3$, thus it is the case that $L_1 = L$, but then either $q_2\in L$ which cannot happen, or $L$ contains $4$ of the $x_i$ points, which contradicts the fact that they span $V$.  Thus no such $q_2$ can exist, and $L$ is the line that satisfies the conclusion.

\end{proof}

\noindent \textbf{Remark:}  When $p=1$, the previous theorem is exactly Theorem 3.12 from \cite{C06}.

\bigskip

We close this section by looking at some examples that will demonstrate the necessity of the assumptions of Theorem 1.1.

\bigskip

\noindent \textbf{Example:}  We will first show that we need two points with Lelong number larger than $\frac{p}{p+1}$.  To see this, let $A:= \{q, x_1, x_2,\dots x_{p+1} \}$ be linearly independent points in $\mathbb{P}^n$, and let $V_j : = \text{span}(A\backslash \{x_j\})$.  Further, let $L_j := \cap_{i=1, i\neq j}^{p+1} V_i$, so $L_j$ is the line spanning $q$ and $x_j$.  Consider the following current:

$$T = \frac{1}{p+1}\sum_{i=1}^{p+1} [V_i]$$

\noindent and note $\nu(T, q) = 1$, and that $q$ is the only point where $T$ has Lelong number strictly larger than $\frac{p}{p+1}$.  Also note along any line $L_j$, $T$ has Lelong number $\frac{p}{p+1}$ and given any $V_i$, there is some line $L_k$ not contained in $V_i$.  Since $\beta < \frac{p}{p+1}$, $E_{\beta}^+(T)$ is not contained in a $p$-dimensional linear subspace, and no matter what line $L$ we consider, $|E_{\beta}^+(T)\backslash L|= \infty$.

\bigskip

\noindent \textbf{Example:}  We now consider the sharpness of our bound $\beta$.  We can only show this for bidimension $(1,1)$ currents, i.e. when $p=1$. Since $p=1$, $\beta = \frac{2-\alpha}{3}$.  Let $L_i$, $i=1,2,3,4$ be complex lines and $q_1, q_2, p_1, p_2$ be points such that $L_1\cap L_3 \cap L_4 =\{q_1\}$, $L_1\cap L_2 = \{q_2\}$, $L_2\cap L_3 = \{p_1\}$, and $L_2\cap L_4 = \{p_2\}$.  We consider the following current

$$T = \frac{7}{15}[L_1] + \frac{6}{15}[L_2] + \frac{1}{15}\sum_{i=3}^{4} [L_i]$$

 \noindent and note $\|T\| = 1$.  Now calculating the Lelong numbers at each of the four previously mentioned points we have:

$$\nu(T, q_1) = \frac{9}{15},\quad \nu(T,q_2) = \frac{13}{15}, \quad \nu(T, p_1) = \frac{7}{15},\quad \nu(T,p_2) = \frac{7}{15}.$$

\smallskip

\noindent  Let $\alpha = \frac{9}{15} = \frac{3}{5}$, noting that $\nu(T, q_i) \geq \alpha$.  Further, $\beta = \frac{7}{15}$, and we observe that $E_{\beta}(T) = L_1\cup\{p_1,p_2\}$, and thus $|E_{\beta}(T)\backslash L_1| = 2$ and $|E_{\beta}(T)\backslash L| = \infty$ for any line $L\neq L_1$.  Finally we observe that $E_{\beta}^+ (T) = \{q_1,q_2\}$, and $|E_{\beta}^+ (T) \backslash L_1| = 0$, showing that the parameter $\beta$ is the best it can be.

\bigskip

We now look at one last interesting example.  In the statement of \cite[Theorem 1.2]{CT15}, Coman mentions that if $E_{\alpha}^+(T)$ is not contained in a $p$-dimensional linear subspace, then the upper level set must be finite.  This however is no longer true in Theorem 1.1 for $E_{\beta}^+(T)$, as we see below.

\bigskip

\noindent \textbf{Example:} Suppose that $A_i$, $i=1,\dots,p+2$ are $p$-dimensional linear subspaces of $\mathbb{P}^n$, $n > p$, such that $L = \bigcap_{i=1}^{p} A_i$.  Let $V$ be the $(p-1)$-dimensional linear space given by $A_{p+1}\cap A_{p+2}$.  Let $\{q_i\} = L \cap A_{p+i}$, $i=1,2$, and $\{x_i\} =  ( \bigcap_{j\neq i, j=1}^{p} A_j  )\cap V$, $i=1,\dots ,p$.  Finally we suppose the $A_i$ are arranged so that $\{q_1, q_2, x_1, \dots, x_p\}$ cannot be contained in a $p$-dimensional linear subspace.  Consider now the following current:

$$T = \frac{1}{p+1} \sum_{i=1}^{p} [A_i] + \frac{1}{2(p+1)} ([A_{p+1}] + [A_{p+2}]).$$

\noindent  It is clear that $\|T\| = 1$, and that $\nu(T, q_i) > \frac{p}{p+1}.$  We note then for any $\alpha$ such that  $\nu(T, q_i) \geq \alpha > \frac{p}{p+1}$, we have $\beta < \frac{p}{p_+1}$.  We now observe that $\nu(T, x_i) = \frac{p}{p+1}$, and give any point $y\in L$, $\nu(T,y) = \frac{p}{p+1}$.  Thus we have that $E_{\beta}^+(T)$ is not contained in a $p$-dimensional linear subspace, $|E_{\beta}^+(T)\backslash L| = p$, however $E_{\beta}^+(T)$ is not finite.

\section{Proof of Theorem 1.2}

Our goal in this section is to attempt to generalize Theorem 2.3 from $\mathbb{P}^2$ to $\mathbb{P}^n$.  In the original proof of Theorem 2.3, we relied heavily on \cite[Theorem 1.2]{C06}, which is only valid in $\mathbb{P}^2$.  Later attempts to generalize \cite[Theorem 1.2]{C06} have yielded good results for all bidimensions except bidimension $(1,1)$, which is sadly the case we would need (see \cite{CT15} for more details).  If we have the situation in which our current has a ``large" Lelong number at four points that are all on a line, then we can avoid the necessity of using  \cite[Theorem 1.2]{C06}, and can generalize it to $\mathbb{P}^n$ with ease.  We now prove Theorem 1.2.

\begin{proof}

Let $L_1$ be the line containing $q_1, q_2, q_3$, and $q_4$. By Siu's decomposition theorem \cite{Siu74} we have that 

$$T = a [L_1 ] + R,$$
\smallskip

\noindent where $a$ is the generic Lelong number of $T$ along $L_1$.  Note that $\| R \| = 1 - a $ and $\nu (R, q_i) \geq \alpha - a$. By \cite{Me98} there is a bidegree $(1,1)$ current $S$ such that $\|S\| = \|R\|$ and $\nu(S,x) = \nu(R,x)$ at all $x$.   Proposition 2.5 shows that there exists a current $S'$ such that $ \|S'\| = 1 - a $, $S'$ is smooth where $S$ has Lelong number $0$, and $\nu (S', q_i) > \alpha  - a -\epsilon$.  By \cite{FS95}, $S' \wedge [L_1] $ is a well defined measure.  Now we have

$$ 1-a = \int_{\mathbb{P}^{n }} S' \wedge [L_1] \geq \sum_{i=1}^{4} \nu (S'\wedge [L_1] , q_i )$$

$$ \geq \sum_{i=1}^{4} \nu (S',q_i) \nu ([L_1], q_i) >4\alpha  - 4 a - 4\epsilon, $$

\smallskip

\noindent where the second inequality follows from \cite[Corollary 5.10]{D93} and the final inequality follows as $\nu([L_1], q_i) = 1$.  So we have that $a \geq \frac{4\alpha - 1}{3}$ as $\epsilon \searrow 0$.

 Define a new current:

$$R' = \frac{R}{1- a} $$

\smallskip

 \noindent and note $R'$ is a bidimension $(1,1)$ current with $\| R' \| = 1$.  Now we have for $x\in E_{\beta}^+ (T)\backslash L_1$:

$$\nu (R', x) > \frac{2}{3} \frac{1-\alpha}{1-a} \geq \frac{2 - 2\alpha}{4-4\alpha} = \frac{1}{2} . $$

\smallskip

Coman's result, \cite[Theorem 1.1]{C06} shows that there is a point $y$ and a line $L_2$ such that $E_{\beta}^+ (T)\backslash L_1 \subset L_2 \cup \{y\}$, and the theorem is proven.
\end{proof}

\smallskip

\noindent \textbf{Remark:} For $n\geq 3$, the two lines need not intersect.

\smallskip

Using Theorem 1.2 combined with \cite[Theorem 3.3]{CT15}, we can get the following partial result:

\begin{Theorem}\label{T:mt1} Let $T$ be a positive closed current of bidimension $(1,1)$ on $\mathbb{P}^n$ with $\|T\|=1$, $\alpha > 2/5$ and $\beta = \frac{2}{3} (1- \alpha)$.  Assume we have one of the following:

\begin{itemize}

\item[i)] $\alpha < 1/2$ and $E_{\alpha}^+(T)$ contains $4$ collinear points, or

\item[ii)] $\alpha < 1/2$, $|E_{\alpha}^+(T)|> 37$, or

\item[iii)] $\alpha \geq 1/2$ and $|E_{\alpha}^+(T)|> 4$

\end{itemize}

\noindent Then there is a curve $C$ in $\mathbb{P}^n$ of degree at most $2$ such that $|E_{\beta}^+(T)\backslash C|\leq 1$.

\end{Theorem}  

\begin{proof} 

(i)  This follows immediately from Theorem 1.2.

(ii) By Theorem 2.4, we know that there is a curve $C_1$ such that  $|E_{\alpha}^+(T)\backslash C_1|\leq 1$.  If $C_1$ omits a point of $E_{\beta}^+(T)$, call it $p_1$, then we can find a point $p_2 \in E_{\beta}^+(T)\backslash C_1$, otherwise we are done.  If $C_1$ omits no points of $E_{\alpha}^+(T)$, then again we can find $p_1,p_2\in E_{\beta}^+(T)\backslash C_1$, otherwise we are done.  In either case, note that $C_1$ must contain at least $38$ points of $E_{\beta}^+(T)$.  We consider the cases of if $C_1$ is an irreducible degree $2$ curve, or not.  

\textit{Case 1:}  First suppose $C_1$ is an irreducible degree $2$ curve.  Let $\alpha '$ be such that $\alpha > \alpha ' > 2/5$, and $\nu (T, p_{i}) > \frac{2}{3} (1- \alpha ') > \beta$, and let $L_{12}$ be the line spanned by $p_1, p_2$. Define a current R as follows:

$$R = \frac{5\alpha ' - 2}{5\alpha '}[L_{12}] + \frac{2}{5\alpha '} T $$ 

\smallskip

\noindent and note $\|{R}\| = 1 $.  We have the following inequalities:

$$\nu (R, q) > \frac{2}{5\alpha '} \alpha > \frac{2}{5}, \; q\in E_{\alpha}^+(T)$$ 

\noindent and

$$ \nu (R, p_{i}) > \frac{5\alpha ' - 2}{5\alpha '} + \frac{4-4\alpha '}{15\alpha '} > \frac{2}{5}, \; i=1,2. $$

\smallskip

\noindent So by the Theorem 2.4, we can find a conic $C_2$ such that $|E_{\alpha}^+(R)\backslash C_2| \leq 1$. Since $C_1$ is irreducible, it is a plane conic by \cite[Proposition 0]{EH87}. If $C_2$ is irreducible as well, then since $|C_1\cap C_2| > 4$, Bezout's theorem show that $C_1 = C_2$, which is impossible as this means one of the $p_i$ are now on $C_1$. If $C_2$ is reducible, then it decomposes into two lines, and can only intersect $C_1$ at most four times, which is a contradiction.

\textit{Case 2:}  If $C_1$ is a reducible conic then $C_1 = L_1\cup L_2$, for some pair of complex lines.  But note then that for one of the lines, say w.l.o.g. $L_1$, $|L_1 \cap E_{\alpha}^+(T)| > 4$, and we have at least four collinear points in $E_{\alpha}^+(T)$, so we are back to situation (i).

(iii)  By \cite[Theorem 1.1]{C06} there is a line $L$ such that $|E_{1/2}^+(T)\backslash L| \leq 1$, so $L$ contains at least four points of $E_{\alpha}^+(T)$, and we are done by Theorem 1.2. 

\end{proof}

\smallskip

\noindent \textbf{Remark:} The span of two non-concurrent lines has dimension 3, which is the reason that the bidimension $(1,1)$ case did not generalize into $\mathbb{P}^n$ (see \cite{CT15} Theorem 1.3, Proposition 3.2, and the remarks following Proposition 3.2 for more details).  We close with the following open question.  Suppose now that $T$ is a positive closed bidimension $(1,1)$ current on $\mathbb{P}^n$.  If we allow for a pair of non-concurrent lines, does there exist a degree 2 curve $C$ such that $|E_{2/5}^+(T)\backslash C| \leq 1$?

\end{document}